\documentclass[12pt,a4paper]{article}

\usepackage{amsfonts,amsthm,amsmath,setspace,graphicx,amssymb,hyperref,mathrsfs}

\newtheorem{theorem}{Theorem}
\newtheorem{lemma}[theorem]{Lemma}

\newtheorem{corollary}[theorem]{Corollary}

\newenvironment{definition}[1][Definition]{\begin{trivlist}
\item[\hskip \labelsep {\bfseries #1}]}{\end{trivlist}}

\newenvironment{remark}[1][Remark]{\begin{trivlist}
\item[\hskip \labelsep {\bfseries #1}]}{\end{trivlist}}

\newcommand{\ds}{\displaystyle}

\newcommand{\mult}{\operatorname{mult}}

\numberwithin{subcase}{case}
\numberwithin{subsubcase}{subcase}
\numberwithin{claim}{theorem}

\title{\bf Rouch\'e's Theorem and the Geometry of Rational Functions}

\author{Trevor J. Richards\footnote{Email: trichards@monmouthcollege.edu}\vspace{6mm}\\{\em Department of Mathematics, Statistics, and Computer Science}\\{\em Monmouth College}\\{\em Monmouth, IL United States}}

\begin{document}
\maketitle

\begin{center}
\end{center}
\begin{center}
\end{center}

\begin{abstract}
In this note, we use Rouch\'e's theorem and the pleasant properties of the arithmetic of the logarithmic derivative to establish several new results regarding the geometry of the zeros, poles, and critical points of a rational function.  Included is an improvement on a result by Alexander and Walsh regarding the distance from a given zero or pole of a rational function to the nearest critical point.
\end{abstract}
\section{Introduction.}%

The logarithmic derivative of the product of two rational functions $g,h\in\mathbb{C}(z)$ may be broken, through the magic of the product rule, into the sum of the logarithmic derivatives: $$\dfrac{(gh)'}{gh}=\dfrac{g'}{g}+\dfrac{h'}{h}.$$  This supplies a ready-made setting for an application of Rouch\'e's theorem.  If we can show that $|g'/g|>|h'/h|$ on some simple closed path $\gamma$ (with bounded face $\Omega$), then Rouch\'e's theorem (along with the observation just made about the logarithmic derivative of a product) tells us that the difference between the number of zeros and the number of poles of $g'/g$ in $\Omega$ equals the difference between the number of zeros and the number of poles of $(gh)'/gh$ in $\Omega$  (see~\cite{C} for Rouch\'e's theorem in this form).  We arrived at this approach for analyzing the geometry of rational functions through working to develop an approximate version of the Gauss--Lucas theorem\footnote{This result is available on the online document posting site \textit{arXiv.org}\cite{R}, with an improvement in publication by Richards and Steinerberger\cite{RS}.  Rouch\'e's theorem also played a role in this latter work.}.  In this article, we will use this method to establish several new results regarding the geometry of rational functions.

\subsection{An exclusion region}

For our first result, we factor a single distinct zero or pole out of a rational function $f$, writing $f(z)=(z-z_0)^kh(z)$, where $k$ is a non-zero integer and $h(z_0)$ is finite and non-zero.  We call the largest punctured ball centered at $z_0$ which contains no zeros, poles, or critical points of $f$ the exclusion region for $f$ around $z_0$.  Applying our logarithmic derivative/Rouch\'e's theorem approach to $f$ thus factored, we obtain a lower bound for the radius of the exclusion region, in terms of the multiplicity of $z_0$ as a zero or pole of $f$, and the distances from $z_0$ to the other zeros and poles of $f$.  In order to discuss these distances efficiently in the appropriate way, we begin with several definitions.

\begin{definition}
For a rational function $f$, we make the following definitions.

\begin{itemize}
\item
Let $Z_f$ and $P_f$ denote the collections of the distinct finite zeros and poles of $f$ respectively.

\item
For $z\in Z_f\cup P_f$, let $\mult_f(z)$ denote the multiplicity of $z$ as a zero or pole of $f$ (positive if a zero, and negative if a pole).
\item
For $z\in\mathbb{C}$, define $d_f(z)$ to be the minimum distance from $z$ to any element of $Z_f\cup P_f\setminus\{z\}$ (that is, the smallest non-zero distance from $z$ to a zero or pole of $f$).
\item
For any point $z\in\mathbb{C}$, define $\rho_f(z)=\ds\sum_{w\in Z_f\cup P_f,w\neq z}\left|\dfrac{\mult_f(w)}{z-w}\right|$.
\end{itemize}

\end{definition}

In 1915, Alexander\cite{A} showed in the case that $f$ is a polynomial, that the exclusion region for $f$ around $z_0$ has radius at least $d_f(z_0)\cdot\mult_f(z_0)/\deg(f)$.  In 1918, Walsh\cite{W} generalized this result to rational functions with a proof based in the context of binary forms, and more recently Alexander's result was re-proven by Khavinson et. al.\cite{K} using matrix theory.  We will prove the following improvement on the Alexander-Walsh result.

\begin{theorem}~\label{thm: Distance to nearest critical point.}
Let $f$ be a rational function with a zero or pole at $z_0\in\mathbb{C}$.  Then $f$ has no critical points in the punctured ball centered at $z_0$ with radius $$\dfrac{d_f(z_0)\cdot|\mult_f(z_0)|}{d_f(z_0)\cdot\rho_f(z_0)+|\mult_f(z_0)|}.$$
\end{theorem}

\begin{remark}
To see that the conclusion of Theorem~\ref{thm: Distance to nearest critical point.} improves on the Alexander-Walsh result, observe that $$d_f(z_0)\cdot\rho_f(z_0)=\ds\sum_{w\in Z_f\cup P_f,w\neq z_0}d_f(z_0)\left|\dfrac{\mult_f(w)}{z_0-w}\right|\leq\sum_{w\in Z_f\cup P_f,w\neq z_0}|\mult_f(w)|,$$  so that we may bound the denominator of our exclusion radius from Theorem~\ref{thm: Distance to nearest critical point.} by $$d_f(z_0)\cdot\rho_f(z_0)+|\mult_f(z_0)|\leq \deg(f),$$ with equality holding only in the case that all of the finite zeros and poles of $f$ (other than $z_0$) lie on the same circle centered at $z_0$.  (Note that in the preceding sentence, $\deg(f)$ denotes the number of finite zeros and poles of $f$, counted with multiplicity: $\deg(f)=\ds\sum_{w\in Z_f\cup P_f}|\mult_f(w)|$.)
\end{remark}

\subsection{A generalization of the exclusion region, with critical point approximation}

Theorem~\ref{thm: Distance to nearest critical point.} states that for $z_0$ a zero or pole of a rational function $f$, a sufficient condition to ensure that $f$ has no critical points in the punctured ball $B^o(z_0;R)$ is the inequality $$R\leq\dfrac{d_f(z_0)\cdot|\mult_f(z_0)|}{d_f(z_0)\cdot\rho_f(z_0)+|\mult_f(z_0)|}.$$

Solving this inequality for $\rho_f$, we obtain
\begin{equation}
\label{eqn: First for rho.}
\rho_f(z_0)\leq|\mult_f(z_0)|\left(\dfrac{1}{R}-\dfrac{1}{d_f(z_0)}\right).
\end{equation}
It is straightforward from the definition of $\rho_f$ and $d_f$ that $1/d_f(z_0)\leq\rho_f(z_0)$.  Using this bound in conjunction with Inequality~\ref{eqn: First for rho.}, after some arithmetic we obtain as a sufficient condition for the function $f(z)=(z-z_0)^kh(z)$ to have no critical points in the punctured ball $B^o(z_0;R)$ the inequality $$\rho_h(z_0)\leq \dfrac{|k|}{(|k|+1)R}.$$  In order to motivate our next theorem, we write this up as a corollary as follows.

\begin{corollary}\label{cor: Motivating corollary.}
For any point $z_0\in\mathbb{C}$, any non-zero integer $k$, and any positive constant $R>0$, if $h$ is a rational function with $$\rho_h(z_0)\leq\dfrac{|k|}{(|k|+1)R},$$ then the function $(z-z_0)^kh(z)$ has no critical points in the punctured ball $B^o(z_0;R)$.
\end{corollary}

If we replace the factor $(z-z_0)^k$ in Corollary~\ref{cor: Motivating corollary.} with a rational function $g(z)$, we see in the following result that by forcing $\rho_h$ to be sufficiently small, we may force some of the critical points of $g(z)h(z)$ to approximate the critical points of $g(z)$, with respect to multiplicity, while the other critical points of $g(z)h(z)$ are excluded from an arbitrarily large disk centered at the origin.

\begin{theorem}\label{thm: Adding in distant zeros and poles.}
For any rational function $g$, any sufficiently large $R>0$, and any sufficiently small $\epsilon>0$, there is a constant $K>0$ such that for any rational function $h$, if $\rho_h(0)<K$ and $h(0)$ is finite and non-zero, then the following holds:
\begin{enumerate}
\item
If $z$ is a critical point of $g$ with multiplicity $m$, then there are exactly $m$ critical points of $g\cdot h$ lying within $\epsilon$ of $z$.

\item
Other than those described in the previous item, there are no critical points of $g\cdot h$ which lie in the disk centered at the origin with radius $R$.
\end{enumerate}
\end{theorem}

Observe that the constant $K$ chosen in Theorem~\ref{thm: Adding in distant zeros and poles.} depends on the choice of the rational function $g(z)$.  In the special case where $g$ is a monic polynomial $p(z)$ with all of its zeros lying in the unit disk, we may find an upper bound for $|p|$ on a disk centered at the origin of radius $R$, namely $(1+R)^{\deg(p)}$.  This allows us to choose the appropriate upper bound for $\rho_h(0)$ depending only on the degree of $p$, but not on $p$ itself.

\begin{theorem}\label{thm: Adding in distant zeros and poles, polynomial case.}
For any positive integer $n$, and any positive constants $R>1$ and $\epsilon>0$, there is a constant $L>0$ such that the following holds.  For any degree $n$ polynomial $p$ having all of its zeros in the unit disk, and any rational function $h$, if $\rho_h(0)<L$ and $h(0)$ is finite and non-zero, then the following holds.
\begin{enumerate}
\item
If $z$ is a critical point of $p$ with multiplicity $m$, then there are at least $m$ critical points of $p\cdot h$ lying within $\epsilon$ of $z$.
\item
There are exactly $n-1$ critical points of $p\cdot h$ lying in the disk centered at the origin with radius $R$.
\end{enumerate}
\end{theorem}

\begin{remark}
Since the polynomial $p$ is chosen after the constant $L$ in the statement of Theorem~\ref{thm: Adding in distant zeros and poles, polynomial case.}, we cannot guarantee that, if $z$ is a critical point of $f$ with multiplicity $m$, there will be exactly $m$ critical points of $p\cdot h$ within $\epsilon$ of $z$, as there may be other critical points of $p$ lying within $\epsilon$ of $z$.  Excluding this possibility is the reason for the ``sufficiently small'' condition on the $\epsilon$ in the statement of Theorem~\ref{thm: Adding in distant zeros and poles.}.
\end{remark}

\subsection{The Critical Points of $g\cdot h^n$}

Theorems and Corollary~\ref{thm: Distance to nearest critical point.}-\ref{thm: Adding in distant zeros and poles, polynomial case.} have each made use of the way in which the logarithmic derivative converts multiplication into addition: $$\dfrac{(gh)'}{gh}=\dfrac{g'}{g}+\dfrac{h'}{h}.$$  The next result uses also the fact that the logarithmic derivative converts exponentiation into scalar multiplication: $$\dfrac{(h^n)'}{h^n}=n\dfrac{h'}{h}.$$  This property is used in the proof of Theorem~\ref{thm: Critical points of gh^n.} below to show that, for any rational functions $g$ and $h$, the non-trivial critical points (that is, those which are not also zeros or poles of the function in question) of the sequence of rational functions $\{g\cdot h^n\}$ approach i) the non-trivial critical point of $h$, with respect to multiplicity, and ii) the zeros and poles of $g$, without respect to multiplicity.

\begin{theorem}\label{thm: Critical points of gh^n.}
For any rational function $g$ and any non-constant rational function $h$, and any sufficiently small $\epsilon>0$, if $n$ is sufficiently large, then the following holds.
\begin{enumerate}
\item
If $z$ is a non-trivial critical point of $h$ with multiplicity $m$, then there are exactly $m$ non-trivial critical points of $g\cdot h^n$ lying within $\epsilon$ of $z$, counted with multiplicity.

\item
If $z$ is a zero or pole of $g$, then there is exactly $1$ non-trivial critical point of $g\cdot h^n$ lying within $\epsilon$ of $z$, counted with multiplicity.
\end{enumerate}

\end{theorem}

\begin{remark}
A simple counting argument shows that if $f(z)$ is a rational function with $k$ distinct zeros and poles, then $f$ has exactly $k-1$ non-trivial critical points, counting multiplicity.  This then implies that the non-trivial critical points of $g\cdot h^n$ described in Items~1 and~2 of Theorem~\ref{thm: Critical points of gh^n.} are in fact all of the non-trivial critical points of $g\cdot h^n$.
\end{remark}

Before proceeding to the proofs, we begin with two helpful lemmas, regarding the function $\rho_f$ and the zeros and poles of the logarithmic derivative of a rational function.

\section{Lemmas.}\label{sect: Lemma.}%

The following lemma gives a quantification of the continuity of $\rho_f$ away from the zeros and poles of $f$.

\begin{lemma}\label{lem: Continuity of rho.}
Let $f$ be a rational function, and let $z_1\in\mathbb{C}$ be neither a zero nor a pole of $f$.  Then for any $\epsilon\in(0,d_f(z_1))$, if $z_2\in\mathbb{C}$ with $|z_2-z_1|<\epsilon$, then $$\rho_f(z_2)<\rho_f(z_1)\dfrac{d_f(z_1)}{d_f(z_1)-|z_2-z_1|}.$$
\end{lemma}

\begin{proof}
Let $\epsilon\in(0,d_f(z_1))$ be chosen.  Let $z_2\in\mathbb{C}$ be chosen such that $|z_2-z_1|<\epsilon$, and let $w$ be a zero or pole of $f$.  Then several applications of the triangle inequality give us
\begin{multline}\label{eqn: Single term difference bound.}
\left|\dfrac{1}{|z_2-w|}-\dfrac{1}{|z_1-w|}\right|\leq\dfrac{|z_1-z_2|}{|z_2-w|\cdot|z_1-w|}\leq\dfrac{1}{|z_1-w|}\dfrac{|z_2-z_1|}{|z_1-w|-|z_2-z_1|}\\<\dfrac{1}{|z_1-w|}\dfrac{|z_2-z_1|}{d_f(z_1)-|z_2-z_1|}.
\end{multline}

Thus applying the triangle inequality and Inequality~\ref{eqn: Single term difference bound.} to $\left|\rho_f(z_2)-\rho_f(z_1)\right|$, we obtain
\begin{multline*}\left|\rho_f(z_2)-\rho_f(z_1)\right|=\left|\ds\sum_{w\in Z_f\cup P_f}\dfrac{1}{|z_2-w|}-\dfrac{1}{|z_1-w|}\right|\\<\sum_{w\in Z_f\cup P_f}\dfrac{1}{|z_1-w|}\dfrac{|z_2-z_1|}{d_f(z_1)-|z_2-z_1|}=\rho_f(z_1)\dfrac{|z_2-z_1|}{d_f(z_1)-|z_2-z_1|}.\end{multline*}

Therefore a final application of the triangle inequality gives us the inequality

\begin{multline*}
\rho_f(z_2)\leq\rho_f(z_1)+|\rho_f(z_2)-\rho_f(z_1)|<\rho_f(z_1)+\rho_f(z_1)\dfrac{|z_2-z_1|}{d_f(z_1)-|z_2-z_1|}\\=\rho_f(z_1)\dfrac{d_f(z_1)}{d_f(z_1)-|z_2-z_1|}.
\end{multline*}
\end{proof}


\begin{lemma}\label{lem: Lemma.}
Let $f$ be a rational function.
\begin{enumerate}
\item
The zeros of the logarithmic derivative $f'/f$ are exactly the non-trivial critical points of $f$, with the same multiplicities.
\item
The poles of $f'/f$ are exactly the zeros and poles of $f$, each with multiplicity one.
\end{enumerate}
\end{lemma}

\begin{proof}
This is a straight forward exercise in arithmetic and the quotient rule for differentiation.
\end{proof}

\section{The Proofs.}\label{sect: Proofs.}

We begin with some notation.

\begin{definition}
Let $X\subset\mathbb{C}$ be a closed set, and let $\iota>0$ be given.
\begin{itemize}
\item We define $$B(X;\iota)=\left\{z\in\mathbb{C}:\min\left(|z-w|:w\in X\right)<\iota\right\}.$$  If $X$ is a singleton $X=\{w\}$, then we just write this set using the normal notation for a disk: $B(w;\iota)$.

\item We define $C(X;\iota)$ to be the boundary of $B(X;\iota)$: $$C(X;\iota)=\left\{z\in\mathbb{C}:\min\left(|z-w|:w\in X\right)=\iota\right\}.$$  Again, if $X$ is a singleton $X=\{w\}$, then we simplify the notation as $C(w;\iota)$.
\end{itemize}

\end{definition}

In each of our proofs, we will employ some version of the following.  For some rational functions $g$ and $h$, some set $X$ (usually some collection of the zeros, poles, and critical points of $g$ and $h$), and some constant $\iota>0$, we show that $|g'/g|>|h'/h|$ on $C(X;\iota)$.  Then Rouch\'e's theorem, along with Lemma~\ref{lem: Lemma.}, tell us that the difference between the number of non-trivial critical points and the number of distinct zeros and poles lying in any given component of $B(X;\iota)$ will be the same for both $g$ and the product $gh$.

We begin with the proof of Theorem~\ref{thm: Distance to nearest critical point.}.

\begin{proof}[Proof of Theorem~\ref{thm: Distance to nearest critical point.}]
Let $f$ be a rational function, and suppose that $z_0$ is either a zero or pole of $f$ (with multiplicity $\mult_f(z_0)$).  Then $f(z)=g(z)h(z)$, where $g(z)=(z-z_0)^{\mult_f(z_0)}$, and $h(z_0)$ is finite and non-zero.  Let $\iota$ represent some small positive number (at least smaller than $d_f(z_0)$).  By the method described at the start of this section, if $|g'/g|>|h'/h|$ on $C(z_0;\iota)$, the difference between the number of non-trivial critical points and the number of distinct zeros and poles lying in $B(z_0;\iota)$ will be the same for both $g$ and $gh$.  Inspecting the function $g$, we see that this common difference equals $-1$.  On the other hand, since the only zero or pole of $gh$ lying in $B(z_0;\iota)$ is $z_0$ (since $\iota<d_f(z_0)$), we conclude that $gh$ has no non-trivial critical points lying in $B(z_0;\iota)$ (our desired conclusion).

Thus we wish to find the largest value of $\iota$ for which we have
\begin{equation}
\label{eqn: First.}
\left|\dfrac{g'(z)}{g(z)}\right|>\left|\dfrac{h'(z)}{h(z)}\right|,
\end{equation}
for all $z\in C(z_0;\iota)$.  If $|z-z_0|=\iota$, then $$\left|\dfrac{g'(z)}{g(z)}\right|=\dfrac{|\mult_f(z_0)|}{\iota},$$ and the triangle inequality, along with Lemma~\ref{lem: Continuity of rho.}, gives \begin{multline*}\left|\dfrac{h'(z)}{h(z)}\right|=\left|\ds\sum_{w\in Z_h\cup P_h,w\neq z}\dfrac{\mult_h(w)}{z-w}\right|\leq\sum_{w\in Z_h\cup P_h,w\neq z}\left|\dfrac{\mult_h(w)}{z-w}\right|=\rho_h(z)\\\leq\rho_h(z_0)\dfrac{d_h(z_0)}{d_h(z_0)-\iota}.\end{multline*}

Thus for $|z-z_0|=\iota$, our goal $|g'/g|>|h'/h|$ is guaranteed by setting $$\rho_h(z_0)\dfrac{d_h(z_0)}{d_h(z_0)-\iota}<\dfrac{|\mult_f(z_0)|}{\iota}.$$  Solving this inequality for $\iota$, we obtain as a sufficient condition for our desired conclusion the inequality $\iota<\dfrac{d_h(z_0)\cdot|\mult_f(z_0)|}{d_h(z_0)\cdot\rho_h(z_0)+|\mult_f(z_0)|}$.  However inspecting the definition of $\rho_f$, $\rho_h$, $d_f$, and $d_h$, we see that $\rho_h(z_0)=\rho_f(z_0)$ and $d_h(z_0)=d_f(z_0)$.  Thus our lower bound on the radius of our exclusion region is the desired one: $\dfrac{d_f(z_0)\cdot|\mult_f(z_0)|}{d_f(z_0)\cdot\rho_f(z_0)+|\mult_f(z_0)|}$.

\end{proof}

We proceed to a proof of Theorem~\ref{thm: Adding in distant zeros and poles.}.

\begin{proof}[Proof of Theorem~\ref{thm: Adding in distant zeros and poles.}]
Let $g$ be a rational function, and let $R>0$ and $\epsilon>0$ be given.  By way of the ``sufficiently small'' and ``sufficiently large'' clauses in the statement of the theorem, we may assume that i) $\epsilon$ is less than one half the minimum distance between any two distinct zeros, poles, or critical points of $g$: $\epsilon<\dfrac{1}{2}\min\left(|w_1-w_2|:w_1,w_2\in Z_g\cup P_g\cup Z_{g'},w_1\neq w_2\right)$, ii) $\epsilon<R$, and iii)$Z_g\cup P_g\cup Z_{g'}\subset B(0;R-\epsilon)$.

By item (i) and (iii) above, $C(Z_g\cup P_g\cup Z_{g'};\epsilon)$ consists of a finite union of non-intersecting circles, each with radius $\epsilon$, centered at the elements of $Z_g\cup P_g\cup Z_{g'}$, all of which is contained in $B(0;R)$.  Define $$K=\dfrac{\min\left(\left|\dfrac{g'(z)}{g(z)}\right|:z\in C(Z_g\cup P_g\cup Z_{g'};\epsilon)\cup C(0;R)\right)}{2},$$ and reduce $K$ further if necessary to ensure that $K<1/2R$.  Let $h$ be any rational function for which $\rho_h(0)<K$ and $h(0)$ is finite and non-zero.

For any point $z\in\mathbb{C}$ with $|z|\leq R$, Lemma~\ref{lem: Continuity of rho.} implies that
\begin{equation}\label{eqn: Bound for rho_h(0).}
\rho_h(z)<\rho_h(0)\dfrac{d_h(0)}{d_h(0)-|z|}.
\end{equation}
It follows immediately from the definition of $\rho_h$ and $d_h$ that $\dfrac{1}{\rho_h(0)}<d_h(0)$, so since $$\rho_h(0)<K<\dfrac{1}{2R},$$ we have $|z|\leq R<\dfrac{d_h(0)}{2}$.  Plugging this into Inequality~\ref{eqn: Bound for rho_h(0).}, we obtain $\rho_h(z)<2\rho_h(0)$.  Now assume further that either $|z|=R$, or $|z-w|=\epsilon$ for some $w\in Z_g\cup P_g\cup Z_{g'}$.  Since $|z|<d_h(0)$, $z$ is not a zero or pole of $h$, so we have $$\left|\dfrac{h'(z)}{h(z)}\right|=\left|\ds\sum_{w\in Z_g\cup P_g}\dfrac{\mult_h(w)}{z-w}\right|\leq\rho_h(z)<2\rho_h(0)<2K<\left|\dfrac{g'(z)}{g(z)}\right|.$$

Let $w_0\in Z_g\cup P_g\cup Z_{g'}$ be given.  We have just seen that $|g'/g|>|h'/h|$ on $C(w_0;\epsilon)$.  By Rouch\'e's theorem and Lemma~\ref{lem: Lemma.}, the difference between the number of non-trivial critical point and the number of distinct zeros and poles lying in $B(w_0;\epsilon)$ is identical for both $g$ and $gh$.

Suppose first that $w_0$ is a zero or pole of $g$.  Since $d_h(0)>R$, and $w_0\in B(0;R-\epsilon)$, $w_0$ is not a zero or pole of $h$, so $w_0$ is a zero or pole of $gh$ with multiplicity $\mult_{gh}(w_0)=\mult_g(w_0)$.  The common difference between the number of non-trivial critical points and the number of distinct zeros and poles lying in $B(w_0;\epsilon)$ equals $-1$ (since no other zero, pole, or critical point of $g$ lies within $\epsilon$ of $w_0$).  Therefore since $w_0$ is a zero or pole of $gh$, we conclude that $gh$ has no non-trivial critical points lying in $B(w_0;\epsilon)$.  This establishes the first item of the theorem for the trivial critical points of $g$.

Now suppose that $w_0$ is a non-trivial critical point of $g$.  Then since no zero or pole of $g$ lies within $\epsilon$ of $w_0$, the common difference between non-trivial critical points and distinct zeros and poles equals $\mult_{g'}(w_0)$.  Since neither $g$ nor $h$ has a zero or pole lying within $\epsilon$ of $w_0$ (again since $d_h(0)>R$), we conclude that $gh$ must have exactly $\mult_{g'}(w_0)$ non-trivial critical points lying in $B(w_0;\epsilon)$, establishing the first item of the theorem also for the non-trivial critical points of $g$.

We have also shown that $|g'/g|>|h'/h|$ on $C(0;R)$.  Since $d_h(0)>R$, $g$ and $gh$ have exactly the same zeros and poles, with the same multiplicities, lying in the disk $B(0;R)$, so Rouch\'e's theorem and Lemma~\ref{lem: Lemma.} imply that $g$ and $gh$ have exactly the same number of non-trivial critical points lying in $B(0;R)$.  This, in conjunction with the result of the first item of the theorem, establishes the second item of the theorem.  (Note that it is important here that the balls $B(w;\epsilon)$ do not overlap for different critical points $w$ of $g$, so that no critical point of $gh$ appears in two such balls.)

\end{proof}

\begin{proof}[Proof of Theorem~\ref{thm: Adding in distant zeros and poles, polynomial case.}]

Fix some positive integer $n$, some $R>1$, and some $\epsilon>0$.  We now turn to the case that $p$ is a polynomial, with $\deg(p)=n$, and with all of the $n$ zeros of $p$ lying in $\mathbb{D}$.  Assume further that $p$ is monic (which is without loss of generality, as multiplying by a non-zero constant has no effect on the locations of the zeros, poles, or critical points of a rational function).  Since all of the zeros of $p$ (and therefore the zeros of $p'$) lie in $\mathbb{D}$, it follows that $$C\left(Z_p\cup Z_{p'};\epsilon/(2(2n-1))\right)\subset\mathbb{D}_{1+\epsilon/(2(2n-1))}.$$  Since $p$ was assumed to be monic, it follows that for $z\in C\left(Z_p\cup Z_{p'};\epsilon/(2(2n-1))\right)$, $$|p(z)|\leq\left(2+\epsilon/(2(2n-1))\right)^n.$$  On the other hand, since the distance from any such $z$ to any zero of $p'$ is at least $\epsilon/(2(2n-1))$, it follows that $$|p'(z)|\geq n\left(\dfrac{\epsilon}{2(2n-1)}\right)^{n-1}.$$  We therefore have the lower bound $$\left|\dfrac{p'(z)}{p(z)}\right|\geq\dfrac{n\left(\dfrac{\epsilon}{2(2n-1)}\right)^{n-1}}{\left(2+\epsilon/(2(2n-1))\right)^n}.$$  Set $$L=\dfrac{1}{2}\dfrac{n\left(\dfrac{\epsilon}{2(2n-1)}\right)^{n-1}}{\left(2+\epsilon/(2(2n-1))\right)^n},$$ and reduce $L$ if necessary to ensure that $$L<\dfrac{1}{2+\epsilon(2(2n-1))}.$$  Let $h$ be any rational function for which $\rho_h(0)<L$ and $h(0)$ is finite and non-zero.

Just as in the proof of Theorem~\ref{thm: Adding in distant zeros and poles.}, Lemma~\ref{lem: Continuity of rho.} now may be applied to a point $z\in C\left(Z_p\cup Z_{p'};\epsilon/(2(2n-1))\right)$, with the eventual conclusion that $$\left|\dfrac{h'(z)}{h(z)}\right|\leq\rho_h(z)<2L\leq\left|\dfrac{p'(z)}{p(z)}\right|.$$  Let $\Omega$ be some component of $B(Z_p\cup Z_{p'};\epsilon/(2(2n-1)))$.  As shown above, $|g'/g|>|h'/h|$ on the boundary of $\Omega$.  $d_h(0)>1/\rho_h(0)>2+\epsilon/(2(2n-1))$ and $\Omega\subset\mathbb{D}_{1+\epsilon/(2(2n-1))}$, so $h$ has no zeros or poles lying in $\Omega$.  Therefore $p$ and $ph$ have precisely the same zeros and poles lying in $\Omega$, so Rouch\'e's theorem and Lemma~\ref{lem: Lemma.} imply that $p$ and $ph$ have the same number of critical points, counted with multiplicity, lying in $\Omega$.

Observe that since $Z_p\cup Z_{p'}$ contains at most $2n-1$ points, and $\Omega$ is a component of $B(Z_p\cup Z_{p'};\epsilon/(2(2n-1)))$, so the diameter of $\Omega$ is at most $$(2n-1)\cdot2\cdot\dfrac{\epsilon}{2(2n-1)}=\epsilon,$$ which establishes the first item of the theorem.

We now turn to the second part of the theorem.  Fix some $z\in\mathbb{C}$ with $|z|=R$.  For the sake of convenience, assume that $z=R$.  Then for any $w=x+iy\in\mathbb{D}$, $\dfrac{1}{z-w}=\dfrac{(R-x)+iy}{(R-x)^2+y^2}$.  Since $y<1$ and $R-1<R-x<R+1$, we have $$\Re\left(\dfrac{1}{z-w}\right)=\dfrac{R-x}{(R-x)^2+y^2}>\dfrac{R-1}{(R+1)^2+1}.$$  Thus we have \begin{multline*}\left|\dfrac{p'(z)}{p(z)}\right|\geq\Re\left(\dfrac{p'(z)}{p(z)}\right)=\displaystyle\sum_{w\in Z_p}\Re\left(\dfrac{\mult_p(w)}{z-w}\right)>\sum_{w\in Z_p}\mult_p(w)\dfrac{R-1}{(R+1)^2+1}\\=n\dfrac{R-1}{(R+1)^2+1}.\end{multline*}  Now reduce $L$ yet further if necessary to ensure that $$L<\dfrac{1}{2}\dfrac{n(R-1)}{(R+1)^2+1}.$$  If $\rho_h(0)<L$, with $h(0)$ finite and non-zero, then by the same method used in the previous part of this proof, and in the proof of Theorem~\ref{thm: Adding in distant zeros and poles.}, we will have $|h'/h|<|p'/p|$ on $C(0;R)$, implying eventually that $p$ and $ph$ have the same number of critical points lying in $\mathbb{D}_R$, establishing the second item of the theorem.

\end{proof}

Our approach in the proof of Theorem~\ref{thm: Critical points of gh^n.} is slightly different.  Now, instead of starting with a rational function $g$, and imposing restrictions on a rational function $h$ to ensure that $|h'/h|<|g'/g|$ on some given set, we start with the functions $g$ and $h$ and a given set, and find a positive integer $n$ which will ensure that $|g'/g|<|(h^n)'/h^n|$ on that set.

\begin{proof}[Proof of Theorem~\ref{thm: Critical points of gh^n.}]

Let $g$ and $h$ be any two rational functions, with $h$ non-constant.  Let $\epsilon>0$ be given.  To simplify our notation, let $A$ denote the collection of all distinct zeros, poles, and critical points of both $g$ and $h$: $A=Z_g\cup P_g\cup Z_{g'}\cup Z_h\cup P_h\cup Z_{h'}$.  Assume that $\epsilon$ is less thanthe minimum distance between any two elements of $A$ (which may be done by the ``sufficiently small'' clause in the statement of the theorem).  Set $$M=\max\left(\left|\dfrac{g'(z)}{g(z)}\right|:z\in\ds\bigcup_{w\in A}C(w;\epsilon)\right),$$ and $$m=\min\left(\left|\dfrac{h'(z)}{h(z)}\right|:z\in\ds\bigcup_{w\in A}C(w;\epsilon)\right).$$

By Lemma~\ref{lem: Lemma.} and the choice of $\epsilon$, $M$ and $m$ are both finite and non-zero.  Choose $n>0$ large enough that $M<nm$.  Fix some $w_0\in A$.  For any $z\in\mathbb{C}$ with $|z-w_0|=\epsilon$, we have $$\left|\dfrac{g'(z)}{g(z)}\right|\leq M<nm<n\left|\dfrac{h'(z)}{h(z)}\right|=\left|\dfrac{(h^n(z))'}{h^n(z)}\right|.$$  Thus by Rouch\'e's theorem and Lemma~\ref{lem: Lemma.}, the difference between the number of non-trivial critical points and the number of distinct zeros and poles lying in $B(w_0;o\epsilon)$ is identical for both $h^n$ and $gh^n$.

Suppose that $w_0$ is a non-trivial critical point of $h$.  Then no zeros or poles of either $g$ or $h^n$ lie in $B(w_0;\epsilon)$, so $h^n$ and $gh^n$ have the same number of non-trivial critical points lying in $B(w_0;\epsilon)$.  However $h$ and $h^n$ have exactly the same non-trivial critical points, with the same multiplicities, so we have established the first item of the theorem.

On the other hand, suppose that $w_0$ is a zero or pole of $g$.  Then again, by choice of $\epsilon$, $h^n$ has no zeros, poles, or critical points lying in $B(w_0;\epsilon)$, and $g$ has no non-trivial critical points in $B(w_0;\epsilon)$.  Therefore the number of non-trivial critical points of $gh^n$ lying in $B(w_0;\epsilon)$ is the same as the number of distinct zeros or poles of $gh^n$ lying in $B(w_0;\epsilon)$, namely one.  This establishes the second item of the theorem.

\end{proof}

\bibliographystyle{amsplain}
\bibliography{ADDITIONAL_FILES/refs}

\end{document}